\documentclass[a4paper,10pt, oneside]{article}
\usepackage{amsmath, amssymb, amsthm, bm, enumerate, mathtools}
\usepackage[font=small,labelsep=period,labelfont=md,textfont=it]{caption}
\usepackage{makecell,multirow,floatrow}
\usepackage{pgfplots}
\usepackage[titletoc, title]{appendix}
\usepackage[colorlinks,linkcolor=blue,citecolor=blue]{hyperref}
\usepackage[capitalise, nosort]{cleveref}

\crefname{equation}{}{}
\crefname{lem}{Lemma}{Lemmas}
\crefname{thm}{Theorem}{Theorems}

\newcommand{\nm}[1]{\left\Vert #1 \right\Vert}
\newcommand{\snm}[1]{\left\vert #1 \right\vert}

\newtheorem{coro}{Corollary}[section]
\newtheorem{Def}{Definition}[section]
\newtheorem{lem}{Lemma}[section]
\newtheorem{rem}{Remark}[section]
\newtheorem{thm}{Theorem}[section]

\pgfplotsset{compat=1.13}

\begin{document}

\title{\Large\bf A time-spectral algorithm for fractional wave problems
\thanks{This work was supported by Major Research Plan of National Natural Science Foundation of China (91430105).}
}

\author{
  Binjie Li \thanks{Email: libinjie@scu.edu.cn},
  Hao Luo \thanks{Corresponding author. Email: galeolev@gmail.com},
  Xiaoping Xie \thanks{Email: xpxie@scu.edu.cn} \\
  {School of Mathematics, Sichuan University, Chengdu 610064, China}
}

\date{}
\maketitle

\begin{abstract}
  This paper develops a high-accuracy algorithm for time fractional wave
  problems, which employs a spectral method in the temporal discretization and a
  finite element method in the spatial discretization. Moreover, stability and
  convergence of this algorithm are derived, and numerical experiments are
  performed, demonstrating the exponential decay in the temporal discretization
  error provided the solution is sufficiently smooth.
\end{abstract}


\section{Introduction}
Let $ 1 < \gamma < 2 $ and let $ \Omega \subset \mathbb R^d $ ($d=2,3$) be a
polygon/polyhedron. This paper considers the fractional wave problem
\begin{equation}
  \label{eq:model}
  \left\{
    \begin{aligned}
      D_{0+}^\gamma (u-u_0-tu_1) - \Delta u &= f &&
      \text{in $ \Omega \times (0,T) $,} \\
      u &= 0 &&
      \text{on $ \partial\Omega \times (0,T) $,} \\
      u(\cdot,0) &= u_0 &&
      \text{in $ \Omega $,} \\
      u_t(\cdot,0) &= u_1 &&
      \text{in $ \Omega $,}
    \end{aligned}
  \right.
\end{equation}
where $ u_0 \in H_0^1(\Omega) $, $ u_1 \in L^2(\Omega) $, and $ f \in
L^2(\Omega_T) $ with $ \Omega_T := \Omega \times (0,T) $. Here $ u_t $ is the
derivative of $ u $ with respect to the time variable $ t $, and $ D_{0+}^\gamma
$ is a Riemann-Liouville fractional differential operator.

The above problem is a particular case of time fractional diffusion-wave
problems, which have attracted a considerable amount of research in the field of
numerical analysis in the past twenty years. By now, most of the existing
numerical algorithms employ the $ L1 $ scheme (\cite{sun2006fully,
Lin2007Finite, Deng2009, Zhang2012COMPACT, zhang2014finite}),
Gr\"unwald-Letnikov discretization (\cite{Chen2007, Meerschaert2004Finite,
Tian2012A, Wang2014, Yuste2005, Yuste2006}) or fractional linear multi-step
method (\cite{Huang2013Two, yang2014numerical, Zeng2013}) to discrete the
fractional derivatives. Generally, for those algorithms, the best temporal
accuracy are $ O(\tau^2) $ for the fractional diffusion problems and $
O(\tau^{3-\gamma}) $ for the fractional wave problems, where $ \tau $ is the
time step size.

Due to the nonlocal property of fractional differential operator, the memory and
computing cost of an accuracy approximation to a fractional diffusion-wave
problem is significantly more expensive than that to a corresponding normal
diffusion-wave problem. To reduce the cost, high-accuracy algorithms are often
preferred, especially those of high accuracy in the time direction. This
motivates us to develop high-accuracy numerical algorithms for problem
\cref{eq:model}. The efforts in this aspect are summarized as follows. Li and Xu
\cite{Li2009} proposed a space-time spectral algorithm for the fractional
diffusion equation, and then Zheng et.~al \cite{Zheng2015A} constructed a high
order space-time spectral method for the fractional Fokker-Planck equation. Gao
et.~al \cite{gao2014new} proposed a new scheme to approximate Caputo fractional
derivatives of order $ \gamma $ ($0<\gamma<1$). Zayernouri and Karniadakis
\cite{Zayernouri2014} developed an exponentially accurate fractional spectral
collocation method. Yang et.~al \cite{yang2016spectral} developed a spectral
Jacobi collocation method for the time fractional diffusion-wave equation.
Recently, Ren et al.~\cite{Ren2017} investigated the superconvergence of finite
element approximation to time fractional wave problems; however, the temporal
accuracy order is only $ O(\tau^{3-\gamma}) $.

In this paper, using a spectral method in the temporal discretization and a
finite element method in the spatial discretization, we design a high-accuracy
algorithm for problem \cref{eq:model} and establish its stability and
convergence. Our numerical experiments show the exponential decay in the
temporal discretization errors, provided the underlying solution is sufficiently
smooth.

The rest of this paper is organized as follows. \cref{sec:nota} introduces some
Sobolev spaces and the Riemann-Liouville fractional calculus operators.
\cref{sec:algor} describes a time-spectral algorithm and constructs the basis
functions for the temporal discretization. \cref{sec:main,sec:proof} establish
the stability and convergence of the proposed algorithm, and \cref{sec:numer}
performs some numerical experiments to demonstrate its high accuracy. Finally,
\cref{sec:conclu} provides some concluding remarks.

\section{Notation}
\label{sec:nota}
Let us first introduce some Sobolev spaces. For $ 0 < \alpha < \infty $, as
usual, $ H_0^\alpha(0,T) $, $ H^\alpha(0,T) $, $ H_0^\alpha(\Omega) $ and $
H^\alpha(\Omega) $ are used to denote four standard Sobolev spaces; see
\cite{Tartar2007}. Let $ X $ be a separable Hilbert space with an inner product
$ (\cdot,\cdot)_X $ and an orthonormal basis $ \{ e_k: k \in \mathbb N \} $.
For $ 0 < \alpha < \infty $, define
\[
  H^\alpha(0,T; X) := \left\{
    v \in L^2(0,T;X):\
    \sum_{k=0}^\infty \nm{(v,e_k)_X}_{ H^\alpha(0,T) }^2 < \infty
  \right\}
\]
and endow this space with the norm
\[
  \nm{\cdot}_{H^\alpha(0,T; X)} := \left(
    \sum_{k=0}^\infty \nm{ (\cdot,e_k)_X }_{ H^\alpha(0,T) }^2
  \right)^{1/2},
\]
where $ L^2(0,T;X) $ is an $ X $-valued Bochner $ L^2 $ space. For $ v \in
H^j(0,T; X) $ with $ j \in \mathbb N_{\geqslant 1} $, the symbol $ v^{(j)} $
denotes its $ j $th weak derivative:
\[
  v^{(j)}(t) := \sum_{k=0}^\infty c_k^{(j)}(t) e_k,
  \quad 0 < t < T,
\]
where $ c_k(\cdot) := (v(\cdot), e_k)_X $ and $ c_k^{(j)} $ is its $ j $th weak
derivative. Conventionally, $ v^{(1)} $ and $ v^{(2)} $ are also abbreviated to
$ v' $ and $ v'' $, respectively.

Moreover, for $ j \in \mathbb N $ we define
\[
  B^j(0,T; X) := \left\{
    v \in L^2(0,T;X):\
    \sum_{k=0}^\infty \nm{(v,e_k)_X}^2_{B^j(0,T)} < \infty
  \right\}
\]
and equip this space with the norm
\[
  \nm{\cdot}_{B^j(0,T; X)} := \left(
    \sum_{k=0}^\infty \nm{(\cdot,e_k)_X}_{B^j(0,T)}^2
  \right)^{1/2},
\]
where the space $ B^j(0,T) $ and its norm are respectively given by
\[
  B^j(0,T) := \left\{
    v \in L^2(0,T):\
    \int_0^T t^i(T-t)^i \snm{v^{(i)}(t)}^2
    \, \mathrm{d}t < \infty,\ 0 \leqslant i \leqslant j
  \right\}
\]
and
\[
  \nm{\cdot}_{B^j(0,T)} := \left(
    \sum_{i=0}^j \int_0^T t^i(T-t)^i
    \snm{(\cdot)^{(i)}(t)}^2 \, \mathrm{d}t
  \right)^{1/2}.
\]

Then we introduce the Riemann-Liouville fractional operators. Let $ X $ be a
Banach space and let $ L^1(0,T;X) $ be an $ X $-valued Bochner $ L^1 $ space.
\begin{Def}
  For $ 0 < \alpha < \infty $, define $ I_{0+}^{\alpha,X},
  I_{T-}^{\alpha,X}:~L^1(0,T;X) \to L^1(0,T;X) $, respectively, by
  \begin{align*}
    \left(I_{0+}^{\alpha,X} v\right)(t) &:=
    \frac1{ \Gamma(\alpha) }
    \int_0^t (t-s)^{\alpha-1} v(s) \, \mathrm{d}s, \quad 0 < t < T, \\
    \left(I_{T-}^{\alpha,X} v\right)(t) &:=
    \frac1{ \Gamma(\alpha) }
    \int_t^T (s-t)^{\alpha-1} v(s) \, \mathrm{d}s, \quad 0 < t < T,
  \end{align*}
  for all $ v \in L^1(0,T;X) $.
\end{Def}
\begin{Def}
  For $ j-1 < \alpha < j $ with $ j \in \mathbb N_{>0} $, define
  \begin{align*}
    D_{0+}^{\alpha,X} & := D^j I_{0+}^{j-\alpha,X}, \\
    D_{T-}^{\alpha,X} & := (-1)^j D^j I_{T-}^{j-\alpha,X},
  \end{align*}
  where $ D $ is the first-order differential operator in the distribution
  sense.
\end{Def}
\noindent Above $ \Gamma(\cdot) $ is the Gamma function, and, for convenience,
we shall simply use $ I_{0+}^\alpha $, $ I_{T-}^\alpha $, $ D_{0+}^\alpha $ and
$ D_{T-}^\alpha $, without indicating the underlying Banach space $ X $. Each $
v \in L^1(\Omega_T) $ also regarded as an element of $ L^1(0,T;X) $ with $
X=L^1(\Omega) $, and thus $ D_{0+}^\alpha v $ and $ D_{T-}^\alpha v $ mean $
D_{0+}^{\alpha,X}v $ and $ D_{T-}^{\alpha,X}v $, respectively, for all $ 0 <
\alpha < \infty $.

\section{Algorithm Definition}
\label{sec:algor}
Let $ \mathcal K_h $ be a triangulation of $ \Omega $ consisting of
$d$-simplexes, and let $ h $ be the maximum diameter of these simplexes in $
\mathcal K_h $. Define
\begin{align*}
  V_h & := \left\{
    v_h \in H^1(\Omega):\
    v_h|_K \in P_m(K) \quad
    \text{for all $ K \in \mathcal K_h $}
  \right\}, \\
  \mathring{V}_h &:= V_h\cap H_0^1(\Omega),
\end{align*}
where $ m $ is a positive integer and $ P_m(K) $ is the set of all polynomials
defined on $ K $ of degree $ \leqslant m $. For $ j \in \mathbb N $, define
\[
  P_j[0,T] \otimes \mathring V_h :=
  \text{span} \big\{
    qv_h:\ v_h \in \mathring V_h,\ q \in P_j[0,T]
  \big\},
\]
where $ P_j[0,T] $ is the set of all polynomials defined on $ [0,T] $ of degree
$ \leqslant j $. Moreover, we introduce a projection operator $ R_h:
H_0^1(\Omega) \to \mathring V_h $ by
\[
  \big( \nabla (I - R_h) v, \nabla
  v_h \big)_{L^2(\Omega)} = 0, \quad
  \forall v \in H_0^1(\Omega), \ \forall v_h \in \mathring V_h.
\]

Now, let us describe a time-spectral algorithm for problem \cref{eq:model} as
follows: seek $ U \in P_M[0,T] \otimes \mathring V_h $ with $ U(0) = R_h u_0 $
such that
\begin{equation}
  \label{eq:U}
  \left(
    D_{0+}^{\gamma_0} (U' - u_{h,1}),
    D_{T-}^{\gamma_0} V
  \right)_{ L^2(\Omega_T) } +
  ( \nabla U,\nabla V )_{ L^2(\Omega_T) } =
  (f,V)_{ L^2(\Omega_T) }
\end{equation}
for all $ V \in P_{M-1}[0,T] \otimes \mathring V_h $, where $ M \geqslant 2 $
is an integer, $ \gamma_0 := (\gamma-1)/2 $, and $ u_{h,1} $ is the $
L^2(\Omega) $-projection of $ u_1 $ onto $ V_h $.
\begin{rem}
  It is well known that the solution to problem \cref{eq:model} generally has
  singularity in time, caused by the fractional derivative. However, in view of
  the basic properties of the operator $ D_{0+}^\gamma $, it is anticipated that
  we can improve the performance of the above algorithm by enlarging $ P_M[0,T]
  $ and $ P_{M-1}[0,T] $ by some singular functions, such as $ t^\gamma $ for $
  P_M[0,T] $ and correspondingly $ t^{\gamma-1} $ for $ P_{M-1}[0,T] $.
\end{rem}

The remainder of this section is devoted to the construction of the bases of $
P_M[0,T] $ and $ P_{M-1}[0,T] $, which is crucial in the implementation of the
proposed algorithm. To this purpose, let us first introduce the well-known
Jacobi polynomials; see \cite{Canuto2006,Shen2011} for more details. Given $ -1
< \alpha, \beta < \infty $, the Jacobi polynomials $ \{ J_n^{(\alpha,\beta)}:\ n
\in \mathbb N \} $ are defined by
\[
  J_n^{(\alpha,\beta)} = w^{-\alpha,-\beta}
  \frac{(-1)^n}{2^n n!} \frac{\mathrm{d}^n}{\mathrm{d}t^n}
  w^{n+\alpha,n+\beta}, \quad n \in \mathbb N,
\]
where
\[
  w^{r,s}(t) := (1-t)^{r} (1+t)^{s},
  \quad -1 < t < 1,
\]
for all $ -\infty < r,s < +\infty $. They form a complete orthogonal basis of $
L_{w^{\alpha,\beta}}^2(-1,1) $, the weighted $ L^2 $ space with weight function
$ w^{\alpha,\beta} $.

Then we construct a basis $ \{p_i\}_{i=0}^M $ of $ P_M[0,T] $ and a basis $
\{q_j\}_{j=0}^{M-1} $ of $ P_{M-1}[0,T] $, respectively, by
\[
  \left\{
    \begin{aligned}
      & p_0(t) := 1, \\
      & p_i(t) := \frac{2t}T J_{i-1}^{(-\gamma_0,0)} \left( 2t/T-1 \right),
      \quad 1 \leqslant i \leqslant M,
    \end{aligned}
  \right.
\]
and
\[
  q_j(t) = J^{(0,-\gamma_0)}_{j}\left(2t/T-1\right),
  \quad 0 \leqslant j \leqslant M-1.
\]
The starting point of the construction of the above two bases is the calculation
of
\begin{equation}
  \label{eq:int}
  \int_0^T D_{0+}^{\gamma_0} p_i' D_{T-}^{\gamma_0} q_j \, \mathrm{d}t.
\end{equation}
To see this, let us first set
\[
  C_{ij} :=
  \begin{cases}
    0, & i = 0, \\
    \frac2T
    \frac{ \Gamma(i)\Gamma(j+1) }{
      \Gamma( j+1-\gamma_0 )\Gamma(i-\gamma_0)
    }, & i \geqslant 1,
  \end{cases} \quad
  D_{ij} :=
  \begin{cases}
    0, & 0 \leqslant i \leqslant 1, \\
    \frac{\Gamma(i+1-\gamma_0)}{\Gamma(i-\gamma_0)T} C_{ij},
    &  i \geqslant 2.
  \end{cases}
\]
By \cite[Lemma~2.5]{chen2016generalized} a straightforward computing yields
\[
  D_{0+}^{\gamma_0} p'_i(t) D_{T-}^{\gamma_0} q_j(t)  =
  t^{-\gamma_0}(T-t)^{-\gamma_0}\zeta_{ij}(t) +
  t^{1-\gamma_0}(T-t)^{-\gamma_0}\varsigma_{ij}(t),
\]
where $ \zeta_{ij}(t) $ and $ \varsigma_{ij}(t) $ are given respectively by
\begin{align*}
  \zeta_{ij}(t) &= C_{ij} \left(
    J_{i-1}^{(0,-\gamma_0)} J_{j}^{(-\gamma_0,0)}
  \right) (2t/T-1), \\
  \varsigma_{ij}(t) &= D_{ij}
  \left(
    J_{i-2}^{(1,1-\gamma_0)} J_{j}^{(-\gamma_0,0)}
  \right) (2t/T-1).
\end{align*}
Then we evaluate \cref{eq:int} precisely by a suitable Jacobi-Gauss quadrature
rule.

\begin{rem}
  It is natural to use
  \[
  \left\{ t^{i}:\ 0 \leqslant i \leqslant M \right\}
  \text{ and }
  \left\{ (T-t)^{j}:\ 0 \leqslant j \leqslant M-1 \right\}
  \]
  as the bases of $ P_M[0,T] $ and $ P_{M-1}[0,T] $ respectively, and in this
  case integral \cref{eq:int} is significantly easier to evaluate. However, as
  the polynomial degree $ M $ increase, the conditioning of the system arising
  from the proposed algorithm deteriorates dramatically, and thus the numerical
  solution becomes unreliable.
\end{rem}

\section{Main Results}
\label{sec:main}
Let us first introduce the following conventions: $ u $ is the solution to
problem \cref{eq:model} and $ U $ is its numerical approximation obtained by the
proposed algorithm; unless otherwise specified, $ C $ is a generic positive
constant that is independent of any function and is bounded as $ M \to \infty $
in each of its presence; $ a \lesssim b $ means that there exists a positive
constant $ c $, depending only on $ \gamma $, $ T $, $ \Omega $, $ m $ or the
shape regular parameter of $ \mathcal K_h $, such that $ a \leqslant c b $; the
symbol $ a \sim b $ means $ a \lesssim b \lesssim a $. The above shape regular
parameter of $ \mathcal K_h $ means
\[
\max\left\{
    h_K/\rho_K:\
    K \in \mathcal K_h
  \right\},
\]
where $ h_K $ is the diameter of $ K $, and $ \rho_K $ is the diameter of the
circle ($ d=2 $) or ball ($ d =3 $) inscribed in $ K $.

Then we introduce an interpolation operator. Let $ X $ be a separable Hilbert
space and let $ P_M[0,T;X] $ be the set of all $ X $-valued polynomials defined
on $ [0,T] $ of degree $ \leqslant M $. Define the interpolation operator
\[
  Q_M^X: H^{1+\gamma_0}(0,T;X) \to P_M[0,T;X]
\]
as follows: for each $ v \in H^{1+\gamma_0}(0,T;X) $, the interpolant $ Q_M^X v
$ fulfills
\[
  \left\{
    \begin{aligned}
      & \left( Q_M^X v \right)(0) = v(0), \\
      & \int_0^T
      D_{0+}^{\gamma_0} \left( v-Q_M^Xv \right)'
      D_{T-}^{\gamma_0} q
      \, \mathrm{d}t = 0, \quad \forall q \in P_{M-1}[0,T].
    \end{aligned}
  \right.
\]
For convenience, we shall use $ Q_M $ instead of $ Q_M^X $ when no confusion
will arise.

\begin{rem}
  Let $ \{e_k: k \in \mathbb N \} $ be an orthonormal basis of $ X $. For any $
  v \in H^{\gamma_0}(0,T;X) $, the definition of $ H^{\gamma_0}(0,T;X) $ implies
  that
  \[
    (v,e_k)_X \in H^{\gamma_0}(0,T) \quad \text{for each $ k \in \mathbb N $},
  \]
  and hence, as \cref{lem:core} (in the next section) indicates
  \[
    \nm{ D_{0+}^{\gamma_0,\mathbb R}(v,e_k)_X }_{L^2(0,T)}
    \sim \nm{(v,e_k)_X}_{H^{\gamma_0}(0,T)},
  \]
  it is evident that
  \[
    \nm{D_{0+}^{\gamma_0,X} v}_{L^2(0,T;X)} =
    \left(
      \sum_{k=0}^\infty
      \nm{D_{0+}^{\gamma_0,\mathbb R} (v,e_k)_X}_{L^2(0,T)}^2
    \right)^\frac12 \sim
    \nm{v}_{H^{\gamma_0}(0,T;X)}.
  \]
\end{rem}

\begin{rem}
  Since $ Q_M^\mathbb R $ is well-defined by \cref{lem:core}, $ Q_M^X $ is
  evidently also well-defined and
  \[
    Q_M^X v = \sum_{k=0}^\infty Q_M^\mathbb R (v,e_k)_X e_k,
    \quad \forall v \in H^{1+\gamma_0}(0,T;X).
  \]
  Furthermore, we can redefine $ Q_M^X $ equivalently as follows: for each $ v
  \in H^{1+\gamma_0}(0,T;X) $, the interpolant $ Q_M^X v $ fulfills
  \[
    \left\{
      \begin{aligned}
        & \left( Q_M^X v \right)(0) = v(0), \\
        & \int_0^T
        \left(
          D_{0+}^{\gamma_0} \left( v-Q_M^Xv \right)',
          D_{T-}^{\gamma_0} q
        \right)_X \, \mathrm{d}t = 0,
        \quad \forall q \in P_{M-1}[0,T;X].
      \end{aligned}
    \right.
  \]
\end{rem}

Finally, we are ready to state the main results of this paper as follows.
\begin{thm}
  \label{thm:stability}
  Algorithm 1 has a unique solution $ U $. Moreover,
  \begin{equation}
    \label{eq:stability}
    \begin{split}
      {} &
      \nm{U}_{ H^{1+\gamma_0}( 0,T; L^2(\Omega) ) } +
      \nm{U(T)}_{ H_0^1(\Omega) } \\
      \lesssim{} &
      \nm{u_0}_{H_0^1(\Omega)} +
      \nm{u_1}_{L^2(\Omega)} +
      \nm{f}_{ L^2(\Omega_T) }.
    \end{split}
  \end{equation}
\end{thm}

\begin{thm}
  \label{thm:conv}
  If $ u \in H^2\left( 0,T; H_0^1(\Omega) \cap H^2(\Omega) \right) $, then
  \begin{align}
    & \nm{u - U}_{ H^{1+\gamma_0}( 0,T; L^2(\Omega) ) }
    \lesssim \eta_1 + \eta_2 + \eta_3 + \eta_4,
    \label{eq:main-1} \\
    & \nm{(u-U)(T)}_{ H_0^1(\Omega) }
    \lesssim \eta_1 + \eta_2 + \eta_3 + \eta_5,
    \label{eq:main-2}
  \end{align}
  where
  \begin{align*}
    & \eta_1 := \nm{ u_1 - u_{h,1} }_{L^2(\Omega)}, \\
    & \eta_2 := CM^{-1-2\gamma_0} \nm{ (I-Q_M)\Delta u }_{
      H^{1+\gamma_0}(0,T;L^2(\Omega))
    }, \\
    & \eta_3 := \nm{(I-R_h)u}_{ H^{1+\gamma_0}( 0, T; L^2(\Omega) ) }, \\
    & \eta_4 := \nm{(I-Q_MR_h)u}_{H^{1+\gamma_0}(0,T;L^2(\Omega))}, \\
    & \eta_5 := \nm{ (u - Q_MR_hu)(T) }_{ H_0^1(\Omega) }.
  \end{align*}
\end{thm}

\begin{coro}
  \label{coro:conv}
  If
  \begin{align*}
    &u \in H^2(0,T;H_0^1(\Omega) \cap H^2(\Omega))
    \cap H^{1+\gamma_0}( 0,T;H^{m+1}(\Omega)), \\
    &u'' \in B^r(0,T;H_0^1(\Omega) \cap H^2(\Omega)),
  \end{align*}
  then
  \begin{align}
    & \nm{u - U}_{ H^{1+\gamma_0} ( 0,T; L^2(\Omega) ) }
    \lesssim \xi_1 + \xi_2 + \xi_3 + \xi_4, \\
    & \nm{u(T) - U(T)}_{ H_0^1(\Omega) }
    \lesssim \xi_1 + \xi_2 + \xi_3 + \xi_5,
  \end{align}
  where $ r \in \mathbb N $ and
  \begin{align*}
    \xi_1 &:= h^{m+1} \nm{u_1}_{ H^{m+1}(\Omega) }, \\
    \xi_2 &:= CM^{-\gamma_0-2-r} \nm{u''}_{ B^r( 0,T; H^2(\Omega) ) }, \\
    \xi_3 &:= h^{m+1} \nm{u}_{H^{1+\gamma_0}(0,T; H^{m+1}(\Omega))}, \\
    \xi_4 &:= CM^{ \gamma_0-1-r } \nm{u''}_{ B^r( 0,T; L^2(\Omega) ) } +
    h^{m+1} \nm{u}_{ H^{1+\gamma_0}( 0,T; H^{m+1}(\Omega) ) }, \\
    \xi_5 &:= CM^{-1.5-r} \nm{u''}_{ B^r( 0,T; H_0^1(\Omega) ) } +
    h^m \nm{u(T)}_{ H^{m+1}(\Omega) }.
  \end{align*}
\end{coro}

\section{Proofs}
\label{sec:proof}
\subsection{Preliminaries}
\label{ssec:preli}
Let us first summarize some standard results.
\begin{lem}
  \label{lem:R_h}
  If $ v \in H_0^1(\Omega) \cap H^{m+1}(\Omega) $, then
  \[
    \nm{(I-R_h)v}_{L^2(\Omega)} +
    h\nm{(I-R_h)v}_{ H_0^1(\Omega) } \lesssim
    h^{m+1} \nm{v}_{ H^{m+1}(\Omega) }.
  \]
\end{lem}

\begin{lem}
  \label{lem:spectral_err}
  If $ v \in H^\alpha(0,T) $ with $ \alpha > \gamma_0 $, then
  \[
    \inf_{q \in P_{M-1}[0,T]} \nm{v-q}_{H^{\gamma_0}(0,T)}
    \leqslant C M^{\gamma_0-\alpha} \nm{v}_{H^\alpha(0,T)}.
  \]
  If $ v \in H^2(0,T) $ such that $ v'' \in B^j(0,T) $ with $ j \in \mathbb N $,
  then
  \[
    \inf_{ q \in P_{M-1}[0,T] } \nm{v-q}_{ H^{1+\gamma_0}(0,T) }
    \leqslant C M^{ \gamma_0 - 1 - j } \nm{v''}_{B^j(0,T)}.
  \]
\end{lem}

\begin{lem}
  \label{lem:basic-frac}
  The following properties hold:
  \begin{itemize}
    \item If $ 0 < \alpha, \beta < \infty $, then
      \[
        I_{0+}^\alpha I_{0+}^\beta = I_{0+}^{\alpha+\beta}, \quad
        I_{T-}^\alpha I_{T-}^\beta = I_{T-}^{\alpha+\beta}.
      \]
    \item If $ 0 < \alpha < \infty $, then
      \[
        \nm{I_{0+}^\alpha v}_{L^2(0,T)} \leqslant C \nm{v}_{L^2(0,T)}, \quad
        \nm{I_{T-}^\alpha v}_{L^2(0,T)} \leqslant C \nm{v}_{L^2(0,T)},
      \]
      where $ C $ is a positive constant that only depends on $ \alpha $ and $ T
      $.
    \item If $ 0 < \alpha < \infty $ and $ u,v \in L^2(0,T) $, then
      \[
        ( I_{0+}^\alpha u, v )_{L^2(0,T)} =
        ( u, I_{T-}^\alpha v )_{L^2(0,T)}.
      \]
  \end{itemize}
\end{lem}

\begin{lem}
  \label{lem:core}
  If $ v \in H^{\gamma_0}(0,T) $, then
  \begin{align*}
    \nm{v}_{H^{\gamma_0}(0,T)} \sim
    \nm{{D_{0+}^{\gamma_0}}v}_{L^2(0,T)} \sim
    \nm{ {D_{T-}^{\gamma_0}} v}_{L^2(0,T)} \sim
    \sqrt{
      \left(
        {D_{0+}^{\gamma_0}}v, {D_{T-}^{\gamma_0}} v
      \right)_{L^2(0,T)}
    }.
  \end{align*}
\end{lem}

\begin{lem}
  \label{lem:AQ}
  Let $ X $ and $ Y $ be two separable Hilbert spaces, and let $ A: X \to Y $ be
  a bounded linear operator. If $ v \in H^{1+\gamma_0}(0,T; X) $, then
  \[
    A Q_M^X v = Q_M^Y Av.
  \]
\end{lem}
\noindent \cref{lem:R_h} is standard \cite{Ciarlet2002}, and, by
\cite[Theorems~3.35--3.37]{Shen2011} and the basic properties of the
interpolation spaces, \cref{lem:spectral_err} is trivial. The proof of
\cref{lem:basic-frac} is included in \cite{Samko1993, Podlubny1998}, and this
lemma will be used implicitly in the forthcoming analysis for convenience.
\cref{lem:core} is a direct consequence of~\cite[Lemma~2.4, Theorem~2.13 and
Corollary~2.15]{Ervin2006}. Finally, by \cref{lem:core} and the basic properties
of the interpolation spaces and the Bochner integrals, a rigorous proof of
\cref{lem:AQ} is tedious but straightforward, and so it is omitted here.

Then let us state three crucial lemmas as follows.
\begin{lem}
  \label{lem:base}
  If $ v \in H^2(0,T) $ and $ w \in H^1(0,T) $, then
  \begin{equation}
    \label{eq:base}
    \left(
      {D_{0+}^\gamma} (v - v(0) - tv'(0), w
    \right)_{L^2(0,T)} =
    \left(
      {D_{0+}^{\gamma_0}} (v' - v'(0)),
      {D_{T-}^{\gamma_0}} w
    \right)_{L^2(0,T)}.
  \end{equation}
\end{lem}

\begin{lem}
  \label{lem:jm}
  If $ v \in H^2(0,T) $ and $ w \in H^{\gamma_0}(0,T) $, then
  \begin{equation}
    \label{eq:jm}
    \big( (I-Q_M)v, w \big)_{L^2(0,T)} \lesssim
    C M^{-1-2\gamma_0} \nm{(I-Q_M)v}_{H^{1+\gamma_0}(0,T)}
    \nm{w}_{H^{\gamma_0}(0,T)}.
  \end{equation}
\end{lem}

\begin{lem}
  \label{lem:I-Q}
  If $ v \in H^2(0,T) $ and $ v'' \in B^j(0,T) $ with $ j \in \mathbb N $, then
  \begin{align}
    \nm{(I-Q_M)v}_{H^{1+\gamma_0}(0,T)} & \lesssim
    C M^{\gamma_0-1-j} \nm{v''}_{B^j(0,T)},
    \label{eq:I-Qfrac} \\
    \nm{(I-Q_M)v}_{L^2(0,T)} & \lesssim
    C M^{-2-j} \nm{v''}_{B^j(0,T)},
    \label{eq:I-Q0} \\
    \nm{(I-Q_M)v}_{C[0,T]} & \lesssim
    C M^{-1.5-j} \nm{v''}_{B^j(0,T)}.
    \label{eq:I-Qinf}
  \end{align}
\end{lem}

Observing that if $ v \in H^2(0,T) $ then a direct calculation yields
\[
  D_{0+}^\gamma (v-v(0) - tv'(0)) = D_{0+}^{\gamma-1} (v'-v'(0)),
\]
we easily see that \cref{lem:base} is a direct consequence of
\cite[Lemma~2.6]{Li2009}. It remains, therefore, to prove \cref{lem:jm,lem:I-Q}.
To this purpose, let us first prove the following lemma.
\begin{lem}
  \label{lem:regu}
  If $ v \in L^2(0,T) $, then
  \begin{equation}
    \label{eq:regu}
    \nm{I_{T-}^{2\gamma_0} v}_{H^{2\gamma_0}(0,T)} \lesssim \nm{v}_{L^2(0,T)}.
  \end{equation}
\end{lem}
\begin{proof}
  Define
  \[
    w(t) := \frac1{\Gamma(\gamma_0)} \int_t^\infty
    (s-t)^{\gamma_0-1} v(s) \, \mathrm{d}s,
    \quad -\infty < t < \infty,
  \]
  where $ v $ is extended to $ \mathbb R \setminus (0,T) $ by zero. Since $ 0 <
  \gamma_0 < 0.5 $, a routine calculation yields $ w \in L^2(\mathbb R) $, and
  then \cite[Theorem~7.1]{Samko1993} implies
  \[
    \mathcal Fw(\xi) = (-\mathrm{i}\xi)^{-\gamma_0} \mathcal Fv(\xi),
    \quad -\infty < \xi < \infty,
  \]
  where $ \mathcal F: L^2(\mathbb R) \to L^2(\mathbb R) $ is the Fourier
  transform operator, and $ \mathrm{i} $ is the imaginary unit. Therefore, the
  well-known Plancherel Theorem yields
  \[
    \nm{w}_{H^{\gamma_0}(\mathbb R)} \lesssim \nm{v}_{L^2(0,T)}
  \]
  and hence
  \[
    \nm{I_{T-}^{\gamma_0} v}_{H^{\gamma_0}(0,T)} \lesssim
    \nm{v}_{L^2(0,T)}.
  \]
  Furthermore, if $ v \in H_0^1(0,T) $ then
  \[
    \nm{ I_{T-}^{\gamma_0}v }_{ H^{1+\gamma_0}(0, T) }
    \lesssim \nm{v}_{ H_0^1(0,T)},
  \]
  by the evident equality $ ( I_{T-}^{\gamma_0} v )' = I_{T-}^{\gamma_0} v' $.
  Consequently, since $ H_0^{\gamma_0}(0,T) $ coincides with $ H^{\gamma_0}(0,T)
  $ with equivalent norms, applying \cite[Lemma~22.3]{Tartar2007} gives
  \[
    \nm{ I_{T-}^{2\gamma_0} v }_{ H^{2\gamma_0}(0,T) } =
    \nm{ I_{T-}^{\gamma_0} I_{T-}^{\gamma_0} v }_{ H^{2\gamma_0}(0,T) }
    \lesssim \nm{ I_{T-}^{\gamma_0} v }_{ H_0^{\gamma_0}(0,T) }
    \lesssim \nm{v}_{L^2(0,T)}.
  \]
  This concludes the proof of the lemma.
\end{proof}

\medskip
\noindent
{\bf Proof of \cref{lem:jm}.}
  Let $ g:= (I-Q_M)v $. Since a straightforward calculation yields
  \[
    \left( I_{0+}^{1-\gamma_0}g' \right)(t) =
    \frac{g'(0)}{ \Gamma(2-\gamma_0) } t^{1-\gamma_0} +
    \left( I_{0+}^{2-\gamma_0}g'' \right)(t),
    \quad 0 < t < T,
  \]
  the fact $ \gamma_0 < 0.5 $ indicates that $ I_{0+}^{1-\gamma_0}g' \in
  H^1(0,T) $ with $ (I_{0+}^{1-\gamma_0}g')(0) = 0 $, and then using integration
  by parts gives
  \begin{align*}
    {} &
    \left(
      D_{0+}^{\gamma_0} g', I_{T-}^{1+\gamma_0} w
    \right)_{L^2(0,T)} =
    \left(
      \left( I_{0+}^{1-\gamma_0} g' \right)',
      I_{T-}^{1+\gamma_0} w
    \right)_{L^2(0,T)} \\
    ={} &
    -\left(
      I_{0+}^{1-\gamma_0} g',
      \left( I_{T-}^{1+\gamma_0} w \right)'
    \right)_{L^2(0,T)} =
    \left(
      I_{0+}^{1-\gamma_0} g', I_{T-}^{\gamma_0} w
    \right)_{L^2(0,T)} \\
    ={} &
    \left( g', I_{T-} w \right)_{L^2(0,T)}.
  \end{align*}
  Hence, as the definition of $ Q_M $ implies $ g(0) = 0 $, we obtain
  \[
    \left(
      D_{0+}^{\gamma_0} g', I_{T-}^{1+\gamma_0} w
    \right)_{L^2(0,T)} =
    \left( g', I_{T-}w \right)_{L^2(0,T)} = (g, w)_{L^2(0,T)},
  \]
  which, combined with the evident equality
  \[
    I_{T-}^{1+\gamma_0} w = D_{T-}^{\gamma_0} I_{T-}^{1+2\gamma_0}w,
  \]
  gives
  \[
    \big( g,w \big)_{L^2(0,T)} =
    \left(
      D_{0+}^{\gamma_0} g',
      D_{T-}^{\gamma_0} I_{T-}^{1+2\gamma_0} w
    \right)_{L^2(0,T)}.
  \]
  Therefore, \cref{lem:core}, the definition of $ Q_M $ and the Cauchy-Schwarz
  inequality indicate
  \[
    \big( g, w \big)_{L^2(0,T)} \lesssim
    \nm{g}_{H^{1+\gamma_0}(0,T)}
    \inf_{q \in P_{M-1}[0,T]}
    \nm{I_{T-}^{1+2\gamma_0}w - q}_{H^{\gamma_0}(0,T)}.
  \]

  Clearly, to prove \cref{eq:jm}, by \cref{lem:spectral_err} it suffices to prove
  \[
    \nm{ I_{T-}^{1+2\gamma_0}w }_{ H^{1+3\gamma_0}(0,T) }
    \lesssim \nm{w}_{ H^{\gamma_0}(0,T) },
  \]
  but since
  \[
    \nm{I_{T-}^{1+2\gamma_0} w}_{H^{1+3\gamma_0}(0,T)} \lesssim
    \nm{I_{T-}^{2\gamma_0} w}_{H^{3\gamma_0}(0,T)},
  \]
  we only need to show
  \begin{equation}
    \label{eq:93}
    \nm{I_{T-}^{2\gamma_0} w}_{H^{3\gamma_0}(0,T)} \lesssim
    \nm{w}_{H^{\gamma_0}(0,T)}.
  \end{equation}
  To this end, observe that \cref{lem:regu} gives
  \[
    \nm{I_{T-}^{2\gamma_0} w}_{H^{2\gamma_0}(0,T)} \lesssim
    \nm{w}_{L^2(0,T)}
  \]
  and that if $ w \in H_0^1(0,T) $ then, due to
  \[
    \left( I_{T-}^{2\gamma_0} w \right)' =
    \left( -I_{T-}^{1+2\gamma_0} w' \right)' =
    I_{T-}^{2\gamma_0} w',
  \]
  again \cref{lem:regu} gives
  \[
    \nm{ I_{T-}^{2\gamma_0}w }_{ H^{1+2\gamma_0}(0,T) }
    \lesssim \nm{w}_{H_0^1(0,T)}.
  \]
  Consequently, using \cite[Lemma~22.3]{Tartar2007} yields \cref{eq:93} and thus
  proves \cref{lem:jm}.
\hfill\ensuremath{\blacksquare}

\medskip\noindent
{\bf Proof of \cref{lem:I-Q}.}
  Let us first consider \cref{eq:I-Qfrac}. For each $ p \in P_{M-1}[0,T] $,
  by \cref{lem:core}, the definition of $ Q_M $ and the Cauchy-Schwarz
  inequality, we obtain
  \begin{align*}
    {} &
    \nm{(Q_Mv)'-p}_{ H^{\gamma_0}(0,T) }^2 \\
    \sim{} &
    \Big(
      D_{0+}^{\gamma_0} \big( (Q_Mv)'-p \big),
      D_{T-}^{\gamma_0} \big( (Q_Mv)'-p \big)
    \Big)_{L^2(0,T)} \\
    ={} &
    \Big(
      D_{0+}^{\gamma_0} (v'-p),
      D_{T-}^{\gamma_0} \big( (Q_Mv)'-p \big)
    \Big)_{L^2(0,T)} \\
    \lesssim{} &
    \nm{v'-p}_{H^{\gamma_0}(0,T)}
    \nm{(Q_Mv)'-p}_{H^{\gamma_0}(0,T)},
  \end{align*}
  which indicates
  \[
    \nm{(Q_Mv)'-p}_{ H^{\gamma_0}(0,T) } \lesssim
    \nm{v'-p}_{ H^{\gamma_0}(0,T) }
  \]
  and hence
  \[
    \nm{(v-Q_Mv)'}_{ H^{\gamma_0}(0,T) } \lesssim
    \nm{v'-p}_{ H^{\gamma_0}(0,T) }.
  \]
  Therefore, since the fact $ (v-Q_Mv)(0) = 0 $ implies
  \[
    \nm{(I-Q_M)v}_{H^{1+\gamma_0}(0,T)} \sim
    \nm{(v-Q_Mv)'}_{H^{\gamma_0}(0,T)},
  \]
  using \cref{lem:spectral_err} proves \cref{eq:I-Qfrac}.

  Next let us consider \cref{eq:I-Q0,eq:I-Qinf}. Proceeding as in the proof of
  \cref{lem:jm} gives
  \begin{align*}
    {} &
    \nm{(I-Q_M)v}_{L^2(0,T)}^2 \\
    \lesssim{} &
    \nm{(I-Q_M)v}_{H^{1+\gamma_0}(0,T)}
    \inf_{ q \in P_{M-1}[0,T] }
    \nm{ I_{T-}^{1+2\gamma_0} (I-Q_M)v-q }_{ H^{\gamma_0}(0,T) } \\
    \lesssim{} &
    C M^{-1-\gamma_0} \nm{  (I-Q_M)v }_{H^{1+\gamma_0}(0,T)}
    \nm{(I-Q_M)v}_{L^2(0,T)},
  \end{align*}
  which proves \cref{eq:I-Q0} by \cref{eq:I-Qfrac}. Then, combining
  \cref{eq:I-Qfrac,eq:I-Q0} and applying \cite[Lemma~22.3]{Tartar2007} yield
  \[
    \nm{(I-Q_M)v}_{H^1(0,T)} \lesssim C M^{-1-j} \nm{v''}_{B^j(0,T)},
  \]
  so that, by \cref{eq:I-Q0}, the estimate \cref{eq:I-Qinf} follows from the
  Gagliardo-Nirenberg interpolation inequality, namely,
  \[
    \nm{w}_{C[0,T]} \lesssim
    \nm{w}_{L^2(0,T)}^\frac12 \nm{w}_{H^1(0,T)}^\frac12,
    \quad \forall w \in H^1(0,T).
  \]
  This concludes the proof of \cref{lem:I-Q}.



\hfill\ensuremath{\blacksquare}


\begin{rem}
  Assume that $ P_M[0,T] $ and $ P_{M-1}[0,T] $ are respectively replaced by
  \[
    P_M[0,T] + \left\{ cw^{1+2\gamma_0}:\ c \in \mathbb R \right\}
    \quad\text{and}\quad
    P_{M-1}[0,T] + \left\{ c w^{2\gamma_0}:\ c \in \mathbb R \right\},
  \]
  where $ w(t) := T-t,\ 0 < t < T $. For each $ v \in H^{1+\gamma_0}(0,T) $,
  the definition of $ Q_M $ implies
  \[
    \int_0^T
    D_{0+}^{\gamma_0}(v-Q_Mv)'
    D_{T-}^{\gamma_0} w^{2\gamma_0}
    \, \mathrm{d}t = 0,
  \]
  and then, as in the previous remark, a straightforward computing yields
  \[
    (v-Q_Mv)(T) = 0.
  \]
  Correspondingly, we can improve \cref{coro:conv} by
  \[
    \xi_5 := h^m \nm{u(T)}_{H^{m+1}(\Omega)}.
  \]
\end{rem}

\subsection{Proofs of Theorems 3.1 and 3.2 and Corollary 3.1}
\noindent
{\bf Proof of \cref{thm:stability}.}
  Since \cref{eq:stability} contains the unique existence of $ U $, it suffices
  to prove the former. Observe first that integration by parts yields
  \[
    2 ( \nabla U, \nabla U' )_{L^2(\Omega_T)} =
    \nm{U(T)}_{H_0^1(\Omega)}^2 -
    \nm{U(0)}_{H_0^1(\Omega)}^2
  \]
  and that \cref{lem:core} implies
  \begin{align*}
    \nm{ D_{0+}^{\gamma_0} u_{h,1} }_{ L^2(\Omega_T) } \sim
    \nm{u_{h,1}}_{ H^{\gamma_0}(0,T;L^2(\Omega) }
    \sim \nm{u_{h,1}}_{L^2(\Omega)}, \\
    \left( D_{0+}^{\gamma_0} U', D_{T-}^{\gamma_0} U' \right)_{ L^2(\Omega_T) }
    \sim \nm{U'}_{ H^{\gamma_0}( 0,T; L^2(\Omega) ) }^2
    \sim \nm{D_{T-}^{\gamma_0} U' }_{ L^2(\Omega_T) }^2.
  \end{align*}
  Moreover, the fact that $ u_{h,1} $ is the $ L^2(\Omega)$-projection of $ u_1
  $ onto $ V_h $ gives
  \[
    \nm{u_{h,1}}_{L^2(\Omega)} \leqslant \nm{u_1}_{L^2(\Omega)}.
  \]
  Consequently, by the Cauchy-Schwarz inequality and the Young's inequality with
  $ \epsilon $, inserting $ V := U' $ into \cref{eq:U} yields
  \begin{align*}
    {} &
    \nm{U'}_{ H^{\gamma_0}( 0,T; L^2(\Omega) ) } +
    \nm{U(T)}_{ H_0^1(\Omega) } \\
    \lesssim{} &
    \nm{U(0)}_{ H_0^1(\Omega) } +
    \nm{u_1}_{L^2(\Omega)} +
    \nm{f}_{ L^2(\Omega_T) },
  \end{align*}
  which, combined with the estimate
  \[
    \nm{U}_{ H^{1+\gamma_0}( 0,T; L^2(\Omega) ) } \sim
    \nm{U(0)}_{L^2(\Omega)} +
    \nm{U'}_{ H^{\gamma_0}( 0,T; L^2(\Omega) ) },
  \]
  indicates
  \begin{align*}
    {} &
    \nm{U}_{ H^{1+\gamma_0}( 0,T; L^2(\Omega) ) } +
    \nm{U(T)}_{ H_0^1(\Omega) } \\
    \lesssim{} &
    \nm{U(0)}_{ H_0^1(\Omega) } +
    \nm{u_1}_{L^2(\Omega)} +
    \nm{f}_{ L^2(\Omega_T) }.
  \end{align*}
  As the definition of $ R_h $ and the fact $ U(0) = R_hu_0 $ imply
  \[
    \nm{U(0)}_{H_0^1(\Omega)} \leqslant
    \nm{u_0}_{H_0^1(\Omega)},
  \]
  this proves \cref{eq:stability} and thus concludes the proof of
  \cref{thm:stability}.
\hfill\ensuremath{\blacksquare}

\medskip\noindent
{\bf Proof of \cref{thm:conv}.}
  Set $ \rho := (I-Q_MR_h)u $ and $ \theta := U - Q_MR_hu $. By \cref{lem:base}
  and integration by parts, using \cref{eq:model} gives
  \[
    \left(
      D_{0+}^{\gamma_0} (u' - u_1),
      D_{T-}^{\gamma_0} \theta'
    \right)_{L^2(\Omega_T)} +
    (\nabla u, \theta')_{L^2(\Omega_T)} =
    (f, \theta')_{L^2(\Omega_T)},
  \]
  which, together with \cref{eq:U}, yields
  \[
    \left(
      D_{0+}^{\gamma_0} \theta',
      D_{T-}^{\gamma_0} \theta'
    \right)_{L^2(\Omega_T)} +
    ( \nabla \theta,\nabla \theta' )_{L^2(\Omega_T)} =
    \mathbb I_1 + \mathbb I_2 + \mathbb I_3,
  \]
  where
  \begin{align*}
    \mathbb I_1 &:=
    ( \nabla \rho,\nabla \theta' )_{L^2(\Omega_T)}, \\
    \mathbb I_2 &:=
    \left(
      D_{0+}^{\gamma_0} \rho', D_{T-}^{\gamma_0} \theta'
    \right)_{ L^2(\Omega_T) }, \\
    \mathbb I_3 &:=
    -\left(
      D_{0+}^{\gamma_0} ( u_1 - u_{h,1} ),
      D_{T-}^{\gamma_0} \theta'
    \right)_{ L^2(\Omega_T) }.
  \end{align*}
  Moreover, the fact $ \theta(0) = 0 $ gives
  \[
    ( \nabla\theta, \nabla\theta' )_{ L^2(\Omega_T) } =
    \frac12 \nm{\theta(T)}_{ H_0^1(\Omega) }^2
  \]
  by integration by parts, and \cref{lem:core} implies
  \[
    \left(
      D_{0+}^{\gamma_0} \theta', D_{T-}^{\gamma_0} \theta'
    \right)_{L^2(\Omega_T)} \sim
    \nm{\theta'}_{ H^{\gamma_0}( 0,T; L^2(\Omega) ) }^2.
  \]
  Therefore, it follows
  \begin{equation}
    \label{eq:err-esti}
    \nm{\theta'}_{ H^{\gamma_0}( 0,T; L^2(\Omega) ) }^2 +
    \nm{\theta(T)}_{H_0^1(\Omega)}^2 \lesssim
    \mathbb I_1 + \mathbb I_2 + \mathbb I_3.
  \end{equation}

  Let us first estimate $ \mathbb I_1 $. Since $ R_h: H_0^1(\Omega) \to
  \mathring V_h $ and $ -\Delta: H^2(\Omega) \to L^2(\Omega) $ are two bounded
  linear operators, \cref{lem:AQ} implies
  \[
    Q_M R_h u = R_h Q_M u \quad \text{ and } \quad
    Q_M(-\Delta u) = -\Delta Q_M u,
  \]
  so that, by integration by parts and the definition of $ R_h $, a
  straightforward calculation gives
  \begin{align*}
    \mathbb I_1
    &= \int_0^T
    \big( \nabla(I-R_hQ_M) u , \nabla \theta' \big)_{L^2(\Omega)}
    \, \mathrm{d}t \\
    &=
    \int_0^T
    \big( \nabla (I-Q_M) u , \nabla \theta' \big)_{L^2(\Omega)}
    \, \mathrm{d}t \\
    &=
    \int_0^T
    \big( -\Delta (I-Q_M) u, \theta' \big)_{L^2(\Omega)} \\
    &=
    \int_0^T \big( (I-Q_M)(-\Delta u) , \theta' \big)_{L^2(\Omega)}
    \, \mathrm{d}t,
  \end{align*}
  Therefore, \cref{lem:jm} leads to
  \begin{equation}
    \label{eq:I_1}
    \mathbb I_1 \lesssim
    C M^{-1-2\gamma_0} \nm{ (I-Q_M)\Delta u }_{ H^{1+\gamma_0}(0,T;L^2(\Omega)) }
    \nm{\theta'}_{ H^{\gamma_0}( 0,T;L^2(\Omega) ) }.
  \end{equation}

  Next let us estimate $ \mathbb I_2 $ and $ \mathbb I_3 $. The definition of $
  Q_M $ gives
  \[
    \mathbb I_2 =
    \left(
      D_{0+}^{\gamma_0} (u - Q_MR_hu)', D_{T-}^{\gamma_0} \theta'
    \right)_{L^2(\Omega_T)} =
    \left(
      D_{0+}^{\gamma_0} ( u-R_hu )', D_{T-}^{\gamma_0} \theta'
    \right)_{L^2(\Omega_T)},
  \]
  so that the Cauchy-Schwarz inequality and \cref{lem:core} indicate
  \begin{equation}
    \label{eq:I_2}
    \mathbb I_2 \lesssim
    \nm{(I-R_h)u}_{H^{1+\gamma_0}(0,T;L^2(\Omega))}
    \nm{\theta'}_{H^{\gamma_0}(0,T;L^2(\Omega))}.
  \end{equation}
  By the evident estimate
  \[
    \nm{ u_1 - u_{h,1} }_{ H^{\gamma_0}(0,T;\Omega_T) } \sim
    \nm{ u_1 - u_{h,1} }_{L^2(\Omega)},
  \]
  the Cauchy-Schwarz inequality and \cref{lem:core} also yield
  \begin{equation}
    \label{eq:I_3}
    \mathbb I_3  \lesssim
    \nm{ u_1 - u_{h,1} }_{L^2(\Omega)}
    \nm{\theta'}_{ H^{\gamma_0}( 0,T; L^2(\Omega) ) }.
  \end{equation}

  Finally, by the Young's inequality with $ \epsilon $, combining
  \cref{eq:err-esti,eq:I_1,eq:I_2,eq:I_3} gives
  \[
    \nm{\theta'}_{ H^{\gamma_0} ( 0,T; L^2(\Omega) ) } +
    \nm{\theta(T)}_{ H_0^1(\Omega) } \lesssim
    \eta_1 + \eta_2 + \eta_3.
  \]
  Since $ \theta(0) = 0 $ implies
  \[
    \nm{\theta}_{
      H^{1+\gamma_0}( 0,T; L^2(\Omega) )
    } \sim
    \nm{\theta'}_{
      H^{\gamma_0} 0,T; L^2(\Omega) )
    },
  \]
  it follows
  \[
    \nm{\theta}_{ H^{1+\gamma_0} ( 0,T; L^2(\Omega) ) } +
    \nm{\theta(T)}_{ H_0^1(\Omega) } \lesssim
    \eta_1 + \eta_2 + \eta_3.
  \]
  As \cref{eq:main-1,eq:main-2} are evident from the above estimate, this
  concludes the proof of \cref{thm:conv}.
\hfill\ensuremath{\blacksquare}

\medskip\noindent
{\bf Proof of \cref{coro:conv}.}
  It suffices to prove $ \eta_i \lesssim \xi_i $ for all $ 1 \leqslant i
  \leqslant 5 $, where $ \{ \eta_i \}_{i=1}^5 $ are defined in \cref{thm:conv}.
  Observing that $ \eta_1 \lesssim \xi_1 $ is a standard result
  \cite{Ciarlet2002}, that $ \eta_2 \lesssim \xi_2 $ follows from
  \cref{lem:I-Q}, and that $ \eta_3 \lesssim \xi_3 $ follows from
  \cref{lem:R_h}, we only need to prove $ \eta_4 \lesssim \xi_4 $ and $ \eta_5
  \lesssim \xi_5 $.

  Let us first consider $ \eta_4 \lesssim \xi_4 $. By \cref{lem:core}, the
  definition of $ Q_M $ implies
  \[
    \nm{Q_M(I-R_h)u}_{ H^{1+\gamma_0}( 0,T; L^2(\Omega) ) }
    \lesssim \nm{(I-R_h)u}_{ H^{1+\gamma_0}( 0,T; L^2(\Omega) ) },
  \]
  so that \cref{lem:R_h} and \cite[Lemma~22.3]{Tartar2007} yield
  \[
    \nm{Q_M(I-R_h)u}_{ H^{1+\gamma_0}( 0,T; L^2(\Omega) ) }
    \lesssim h^{m+1} \nm{u}_{ H^{1+\gamma_0}( 0,T; H^{m+1}(\Omega) ) }.
  \]
  Moreover, \cref{lem:I-Q} gives
  \[
    \nm{(I - Q_M)u}_{ H^{1+\gamma_0}(0,T; L^2(\Omega)) }
    \lesssim C M^{\gamma_0 -1 - r }
    \nm{u''}_{ B^r(0,T; L^2(\Omega)) }.
  \]
  Consequently, $ \eta_4 \lesssim \xi_4 $ is a direct consequence of the
  inequality
  \begin{align*}
    & \nm{(I-Q_MR_h)u}_{ H^{1+\gamma_0}( 0,T; L^2(\Omega) ) } \\
    \leqslant &
    \nm{(I-Q_M)u}_{ H^{1+\gamma_0}( 0,T; L^2(\Omega) ) } +
    \nm{Q_M(I-R_h)u}_{ H^{1+\gamma_0}( 0,T; L^2(\Omega) ) }.
  \end{align*}

  Then let us consider $ \eta_5 \lesssim \xi_5 $. Since \cref{lem:AQ} gives $
  R_hQ_M u = Q_M R_h u $, the definition of $ R_h $ yields
  \[
    \nm{ (R_hu - Q_MR_hu)(T) }_{ H_0^1(\Omega) } \leqslant
    \nm{(u-Q_Mu)(T)}_{ H_0^1(\Omega) },
  \]
  and hence \cref{lem:I-Q} indicates
  \[
    \nm{(R_hu - Q_MR_hu)(T)}_{H_0^1(\Omega)} \lesssim
    C M^{-1.5-r} \nm{u''}_{B^r(0,T;H_0^1(\Omega))}.
  \]
  Therefore, as \cref{lem:R_h} implies
  \[
    \nm{(I-R_h)u(T)}_{H_0^1(\Omega)} \lesssim
    h^m \nm{u(T)}_{H^{m+1}(\Omega)},
  \]
  the estimate $ \eta_5 \lesssim \xi_5 $ follows from the inequality
  \begin{align*}
    & \nm{(u - Q_MR_hu)(T)}_{H_0^1(\Omega)} \\
    \leqslant &
    \nm{(I - R_h) u(T)}_{H_0^1(\Omega)} +
    \nm{(R_h u - Q_MR_hu)(T)}_{H_0^1(\Omega)}.
  \end{align*}
  This concludes the proof of \cref{coro:conv}.
\hfill\ensuremath{\blacksquare}

\section{Numerical Experiments}
\label{sec:numer}
This section performs some numerical experiments to demonstrate the high order
accuracy of the proposed algorithm in two dimensional case. Throughout this
section we set $\gamma:= 1.5 $, $ T := 1 $ and $ \Omega := (0,1)^2 $.

{\it Example 1.} In this example the solution to problem \cref{eq:model} is
\[
  u(x,t) := t^{20} x_1x_2(1-x_1)(1-x_2),
  \ (x,t) \in \Omega_T,
\]
where $ x = (x_1,x_2) $. Let us first consider the spatial discretization errors
of the proposed algorithm, and, to this end, we set $ M := 20 $ to ensure that
the temporal discretization errors are negligible compared with the former. The
corresponding numerical results, presented in \cref{tab:ex1-x}, illustrate that
the convergence orders of
\[
  \nm{(u-U)(T)}_{H_0^1(\Omega)} \quad \text{ and } \quad
  \nm{u-U}_{H^{1+\gamma_0}(0,T;L^2(\Omega))}
\]
are $ m $ and $ m + 1 $ respectively, which agrees well with \cref{coro:conv}.
Then let us consider the temporal discretization errors and hence set $ m := 4 $
and $ h := 1/32 $ to ensure that the temporal discretization error is dominant.
We present the corresponding numerical results in \cref{tab:ex1-t} and plot the
log-linear relationship between the errors and the polynomial degree $ M $ in
\cref{fig:ex1}. As indicated by \cref{coro:conv}, these numerical results
demonstrate that the errors reduce exponentially as $ M $ increases.

\begin{table}[H]
  \caption{The errors for Example 1 with $ M = 20 $.}
  \label{tab:ex1-x}
  \begin{tabular}{ccccp{0.5cm}cc}
    \hline
    \multirow{2}*{$m$} &
    \multirow{2}*{$1/h$} &
    \multicolumn{2}{c}{$ \nm{u(T)-U(T)}_{H_0^1(\Omega)} $} &&
    \multicolumn{2}{c}{$ \nm{u-U}_{H^{1+\gamma_0}(0,T;L^2(\Omega))} $} \\
    \cline{3-4} \cline{6-7}
    &        &  Error     & Order  &&  Error     & Order   \\
    \hline\multirow{5}*{$1$}
    &  $2$   &  1.19e-01  &  --    &&  8.68e-02  &  --     \\
    &  $4$   &  6.12e-02  &  0.95  &&  1.94e-02  &  2.17   \\
    &  $8$   &  3.06e-02  &  1.01  &&  4.52e-03  &  2.10   \\
    &  $16$  &  1.52e-02  &  1.01  &&  1.10e-03  &  2.03   \\
    &  $32$  &  7.61e-03  &  1.00  &&  2.74e-04  &  2.01   \\
    \hline\multirow {5}*{$2$}
    &  $2$   &  3.12e-02  &  --    &&  1.18e-02  &  --     \\
    &  $4$   &  8.28e-03  &  1.91  &&  1.63e-03  &  2.86   \\
    &  $8$   &  2.11e-03  &  1.97  &&  2.12e-04  &  2.95   \\
    &  $16$  &  5.31e-04  &  1.99  &&  2.67e-05  &  2.98   \\
    &  $32$  &  1.33e-04  &  2.00  &&  3.35e-06  &  3.00   \\
    \hline\multirow {5}*{$3$}
    &  $2$   &  4.92e-03  &  --    &&  1.50e-03  &  --     \\
    &  $4$   &  5.94e-04  &  3.05  &&  9.13e-05  &  4.04   \\
    &  $8$   &  7.28e-05  &  3.03  &&  5.51e-06  &  4.05   \\
    &  $16$  &  9.01e-06  &  3.02  &&  3.36e-07  &  4.04   \\
    &  $32$  &  1.12e-06  &  3.01  &&  2.07e-08  &  4.02   \\
    \hline
  \end{tabular}
\end{table}

\begin{table}[H]
  \renewcommand\arraystretch{1}
  \caption{The errors for Example 1 with $ m=4 $ and $ h=1/32 $.}
  \label{tab:ex1-t}
  \begin{tabular}{cccp{0.5cm}cc}
    \hline
    \multirow{2}*{$ M $} &
    \multicolumn{2}{c}{$ \nm{u(T)-U(T)}_{H_0^1(\Omega)} $} &&
    \multicolumn{2}{c}{$ \nm{u-U}_{H^{1+\gamma_0}(0,T;L^2(\Omega))} $} \\
    \cline{2-3} \cline{5-6}
    & Error      & Order   && Error      & Order  \\
    \hline
    9    &  7.05e-05  &  --     &&  4.13e-03  & --     \\
    11   &  4.48e-06  &  13.74  &&  4.47e-04  & 11.08  \\
    13   &  1.64e-07  &  19.80  &&  2.63e-05  & 16.97  \\
    15   &  3.06e-09  &  27.83  &&  7.28e-07  & 25.06  \\
    17   &  2.10e-11  &  39.80  &&  7.16e-09  & 36.92  \\
    \hline
  \end{tabular}
\end{table}

\begin{figure}[H]
  \centering
  \begin{tikzpicture}
    \begin{axis}[
      small,
      xlabel = polynomial degree $M$,
      ylabel = erros in logscale,
      xtick = {9, 7, 11, 13, 15, 17},
      legend cell align = left,
      legend style = {legend pos = south west, draw=none, font=\tiny}
      ]

      \addplot[smooth,mark=o,color=blue]
      plot coordinates{
        (9, -4.1518)
        (11, -5.3489)
        (13, -6.7848)
        (15, -8.5147)
        (17, -10.6781)
      };
      \addlegendentry{$H_0^1(\Omega)$}

      \addplot[smooth,mark=x,color=red]
      plot coordinates{
        (9, -2.3835)
        (11, -3.3493)
        (13, -4.5805)
        (15, -6.1380)
        (17, -8.1451)
      };
      \addlegendentry{$H^{1+\gamma_0}(0,T;L^2(\Omega))$}
    \end{axis}
  \end{tikzpicture}
  \caption{The log-linear relationship between the errors and the polynomial
    degree $ M $ for Example 1 with $ m=4 $ and $ h=1/32 $. 
  }
  \label{fig:ex1}
\end{figure}
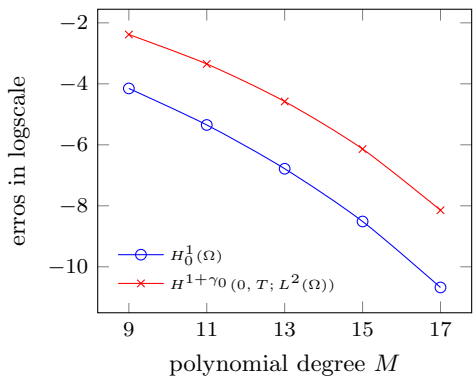

{\it Example 2.} This example adopts
\[
  u(x,t) := t^2 \snm{1-2t}^\beta x_1 (1-x_1) \sin(\pi x_2),
  \quad (x,t) \in \Omega_T
\]
as the solution to problem \cref{eq:model}, where $ \beta $ is a positive
constant. Here we only consider the temporal discretization errors and hence set
$ m := 6 $ and $ h := 2^{-4} $ to ensure that the temporal discretization errors
are dominant. The corresponding numerical results are presented in
\cref{tab:ex2-t1,tab:ex2-t2}. Observing that
\[
  \snm{1-2t}^{\beta} \in H^{\beta+0.5-\epsilon}(0,T)  \quad
  \text{for all $ \epsilon > 0 $ },
\]
by \cref{coro:conv} and \cite[Lemma~22.3]{Tartar2007} we have
\[
  \begin{aligned}
    \nm{u(T)-U(T)}_{ H_0^1(\Omega) } &
    \lesssim C(\epsilon)M^{-\beta+\epsilon} ,\\
    \nm{u-U}_{ H^{1+\gamma_0}( 0,T;L^2(\Omega) ) } &
    \lesssim C(\epsilon)M^{0.75-\beta+\epsilon},
  \end{aligned}
\]
where $ C(\epsilon) $ is a constant that depends on $ \epsilon $. Evidently, for
the convergence order of $ \nm{u-U}_{ H^{1+\gamma_0}( 0,T;L^2(\Omega) ) } $, the
numerical results are in agreement with \cref{coro:conv}. However, in this case,
$ \nm{(u-U)(T)}_{H_0^1(\Omega)} $ reduces significantly faster than that
predicted by \cref{coro:conv}.

\begin{table}[H]
  \renewcommand\arraystretch{1}
  \caption{The errors for Example 2 with $ \beta=2.5 $.}
  \label{tab:ex2-t1}
  \begin{tabular}{cp{0.2cm}ccp{0.2cm}cc}
    \hline
    \multirow{2}*{$M$}  &&
    \multicolumn{2}{c}{$ \nm{u(T)-U(T)}_{H_0^1(\Omega)} $} &&
    \multicolumn{2}{c}{$ \nm{u-U}_{H^{1+\gamma_0}(0,T;L^2(\Omega))} $} \\
    \cline{3-4} \cline{6-7}
    && Error     & Order  && Error      & Order    \\
    \hline
    7    &&  3.80e-5  &  --    &&  3.00e-03  &  --      \\
    9    &&  1.60e-5  &  3.44  &&  1.94e-03  &  1.73    \\
    11   &&  6.32e-6  &  4.63  &&  1.35e-03  &  1.81    \\
    13   &&  2.77e-6  &  4.93  &&  9.94e-04  &  1.84    \\
    15   &&  1.38e-6  &  4.86  &&  7.64e-04  &  1.85    \\
    17   &&  7.40e-7  &  4.99  &&  6.06e-04  &  1.84    \\
    \hline
  \end{tabular}
\end{table}
\begin{table}[H]
  \renewcommand\arraystretch{1}
  \caption{The errors for Example 2 with $ \beta=2.1 $.}
  \label{tab:ex2-t2}
  \begin{tabular}{cp{0.2cm}ccp{0.2cm}cc}
    \hline
    \multirow{2}*{$M$}  &&
    \multicolumn{2}{c}{$ \nm{u(T)-U(T)}_{H_0^1(\Omega)} $} &&
    \multicolumn{2}{c}{$ \nm{u-U}_{H^{1+\gamma_0}(0,T;L^2(\Omega))} $} \\
    \cline{3-4} \cline{6-7}
    &&  Error    &  Order &&  Error     &  Order   \\
    \hline
    7    &&  1.24e-5  &  --    &&  1.05e-03  &  --      \\
    9    &&  5.48e-6  &  3.24  &&  7.49e-03  &  1.36    \\
    11   &&  2.32e-6  &  4.28  &&  5.64e-04  &  1.41    \\
    13   &&  1.08e-6  &  4.56  &&  4.45e-04  &  1.42    \\
    15   &&  5.72e-7  &  4.46  &&  3.63e-04  &  1.43    \\
    17   &&  3.22e-7  &  4.59  &&  3.03e-04  &  1.42    \\
    \hline
  \end{tabular}
\end{table}

\section{Conclusions}
\label{sec:conclu}
In this paper, a high accuracy algorithm for time fractional wave problems is
developed, which adopts a spectral method to approximate the fractional
derivative and uses a finite element method in the spatial discretization.
Stability and a priori error estimates of this algorithm are derived, and
numerical experiments are also performed to verify its high accuracy.

In future work, we shall consider the following issues. Firstly, the optimal
error estimates of $ \nm{(u-U)(T)}_{L^\infty(\Omega)} $ and $
\nm{(u-U)(T)}_{L^2(\Omega)} $ are not established. Secondly, it is worth
applying the idea of approximating fractional differential operators of order $
\gamma $ ($ 1 < \gamma < 2$) by spectral methods to other fractional
differential equations, such as nonlinear fractional ordinary differential
equations and nonlinear time fractional wave equations.


\end{document}